\documentclass[12 pt, a4paper]{amsart}
\usepackage[T1]{fontenc}
\usepackage[latin9]{inputenc}
\usepackage{xcolor}
\usepackage{pdfcolmk}
\usepackage{amsthm}
\usepackage{amstext}
\usepackage[all,cmtip]{xy}
\usepackage{amssymb}
\usepackage{setspace} 

\PassOptionsToPackage{normalem}{ulem}
\usepackage{ulem}

\makeatletter

\providecolor{lyxadded}{rgb}{0,0,1}
\providecolor{lyxdeleted}{rgb}{1,0,0}

\usepackage{graphicx}

\usepackage{amsfonts}
\usepackage{amssymb,latexsym,epsfig,color}
\usepackage{enumerate}
\usepackage[all,cmtip]{xy}
\usepackage{cite}

\newcommand{\TT}{{\mathcal T}}

\newtheorem{theorem}{Theorem}[section]

\newtheorem{corollary}[theorem]{Corollary}

\newtheorem{question}[theorem]{Question}

\theoremstyle{definition}

\theoremstyle{remark}
\newtheorem{rem}[theorem]{Remark}
\begin{document}
\pagestyle{plain}
\title{A family of $\displaystyle \omega_1$ many topological types of locally finite trees.}
\author{Jorge Bruno}
\date{\today}
\maketitle

\begin{abstract}

\vspace{.6 in}

Two rooted locally finite trees are considered equivalent if both can be embedded into each other as topological minors by means of tree-order preserving mappings. By exploiting Nash-William's Theorem, Matthiesen provided a non-constructive proof of the uncountability of such equivalence classes, thus answering a question of van der Holst. As an open problem, Matthiesen asks for a constructive proof of this fact. The purpose of this paper is to provide one such construction; working solely within ZFC we illustrate a collection of $\omega_1$ many topological types of rooted trees. In particular, we also show that this construction strengthens that of Mathiesen in that it also applies to {\it free} (unrooted) trees of degree two.

 \end{abstract}
 
 \section{Introduction}
 
 For two rooted trees $T$ and $U$ we say that $\phi: V(T) \to V(U)$ is an {\it embedding} of $T$ into $U$ if there is an extension of $\phi$ from a subdivision of $T$ to the smallest subtree $U_T$ of $U$ containing all vertices from $\phi(T)$. Of course, there must be no other vertex from $U_T$ between the root of $U$ and $\phi$'s image of $T$'s root. Equivalently, one can define this embedding by using the {\it tree-order}: for a tree $T$, $a\leq b$ provided that $a$ lies in the path from the root of $T$ to $b$. In light of this, an embedding between two trees is then a tree-order preserving function $T \to U$. If any embedding $T \to U$ exists, $T$ is said to be a {\it topological minor} of $U$ and we write $T \preceq U$. It is simple to show that the collection of all locally finite trees with the topological minor relation forms a quasi-ordered set of size $\mathfrak{c}$ (i.e., size continuum). 
When both $T\preceq Y$ and $U\preceq T$ are true, $T$ and $U$ are said to be {\it equivalent} and the equivalence classes generated by this relation are called {\it topological types}. A natural question to ask is:

\begin{question} What is the size of the partially ordered set generated by considering all locally finite trees modulo this topological equivalence?
\end{question}

This question was originally posed by H. van der Holst and partially answered by Matthiesen in \cite{MR2236511} by non-constructive means. More precisely, let $\lambda$ denote the number of topological types of locally finite trees: clearly, $\omega \leq \lambda \leq \mathfrak{c}$ and Matthiesen refined it to $\omega_1 \leq \lambda \leq \mathfrak{c}$ by clever use of Nash-William's Theorem (which states that the infinite rooted trees are better-quasi-ordered under topological minor \cite{MR0175814}, \cite{MR1816801}). A good introduction to the subject can be found in \cite{MR2159259}. Matthiesen leaves as an open problem a constructive proof of this fact. Working solely within ZFC, we address Mathiesen's problem and extend her result by providing a construction of $\omega_1$ many topological types of free locally finite trees. In light of the Continuum Hypothesis (CH) this is the largest construction allowed within ZFC. We address this and other affine issues at the end of the paper. 
 
 \section{The construction}
 
 We will inductively define a family $\TT = \{T_\alpha \mid \alpha < \omega_1\}$ with the property that for all $\alpha < \beta < \omega_1$, $T_\alpha \preceq T_\beta \npreceq T_\alpha$. This collection would then be a specimen of an uncountable collection of topological types. Let's begin by defining $T_1$ to be a ray $R = v_1v_2v_3\ldots $ (denoted simply by $R$ from now on). The tree $T_2$ will be constructed by {\it attaching} (by {\it attach} we mean ``to join with an edge'') to each vertex of $R$ a copy of $T_1$ by its root. The resulting tree resembles a comb whose teeth are copies of $T_1$ and we denote this tree by $T_2$. In general:
 
 \begin{itemize}
 \item For $\alpha +1 =\beta$: $T_\beta$ is forged by a ray $R$ so that each vertex $v_i$ is attached to the root of a - unique - copy of $T_\alpha$. 

\medskip

 \item Whenever $\beta$ is limit: choose any strictly monotone and cofinal $\psi: \omega \to \beta$, for each $v_i$ in the ray $R$ attach a - unique - copy of $T_{\psi(i)}$ by its root. 
  \end{itemize}
 
 \noindent
For each tree $T_\alpha$ the ray used on its construction will be called its {\it spine}. 
 Next we prove our main result.
  
 \begin{theorem}\label{thm:construction} For any $\alpha < \beta < \omega_1$, $T_\alpha \preceq T_\beta \npreceq T_\alpha$.
 \end{theorem}

 \begin{proof} By design it should be clear that for any $\alpha < \beta < \omega_1$, $T_\alpha \preceq T_\beta$. We thus focus on the latter assertion and notice that $T_1 \preceq T_2 \npreceq T_1$ starts the induction. For any $\beta = \alpha +1$ notice that since $\TT$ is nested, we need only show that $T_\beta \npreceq T_\alpha$. If $\beta$ is limit, we still have to show that $T_\beta \npreceq T_\alpha$ for all $\alpha < \beta$. Assume that for a given $\alpha < \beta < \gamma < \omega_1$ we have $T_\alpha \preceq T_\beta \npreceq T_\alpha$.\\
 
 \noindent
 \underline{$\gamma$ is a limit ordinal:}\\
 
 Let $\psi: \omega \to \gamma$ be the strictly monotone and cofinal function defining $T_\gamma$. Assume that for some $\alpha < \gamma$, $T_\gamma \preceq T_\alpha$ and notice that for some $j\in \omega$ we have $ \alpha < \psi(j)$. Since the j$^\text{th}$ branch of $T_\gamma$ is $T_{\psi(j)}$, this then yields $T_{\psi(j)} \preceq T_\alpha$, an impossibility. \\

 \noindent
 \underline{$\gamma = \beta + 1$:}\\
 
 Assume that there is an embedding from $T_\gamma$ into $T_\beta$ and consider any branch $T_\beta^i$ (the i$^{\text{th}}$ copy of $T_\beta$ along $T_\gamma$'s spine). The embedding cannot map $T_\beta^i$ strictly within any branch that stems off of $T_\beta$'s spine (recall that by induction $T_\beta$ cannot be embedded into any such branch). However, if any vertex of $T_\beta^i$ is mapped to a vertex, $p$, on the spine of $T_\beta$ then the since the embedding preserves tree-order (i.e., incomparable vertices are mapped to incomparable vertices) no branch in $T_\gamma$ higher up than $T_\beta^i$ can be mapped anywhere above $p$. This leaves only finitely many branches in $T_\beta$ for mapping the rest of $T_\gamma$. Yielding that at least one branch of $T_\gamma$ will be mapped within a branch of $T_\beta$ and another contradiction.

 \end{proof}

\begin{corollary} The family $\TT$ contains $\omega_1$ many topological types of free locally finite trees.
\end{corollary}

\begin{proof} It is evident from the proof of Theorem~\ref{thm:construction} that even considered as free trees, the family $\TT$ contains $\omega_1$ many topological types. 
\end{proof}
 
\begin{rem} Due Nash-Williams Theorem we are forced to construct well-ordered chains when searching for large families of topological types of locally finite trees; the {\it  width} of any such family must be finite and thus the bulk of its cardinality must be derived from its height.
\end{rem}

 \section{Conclusions}
 
 Working within ZFC and due to CH, in terms of cardinality, one cannot construct a larger example than the one presented here. It remains an open question whether or not it is a ZFC theorem that any family of topological types of locally finite trees must have cardinality of at most $\omega_1$.

\begin{question} Is it a theorem of ZFC that there does not exists a family of topological types whose size exceeds $\omega_1$?
\end{question}

 It is simple to show that any tree that contains a copy of all trees in $\TT$ must have a copy of the full binary tree. By our closing remark above, this then suggest that the bulk of any potentially large family of topological types (consistent with ZFC) must be developed from well-ordered chains of trees containing the full binary tree. More precisely we have the following questions.

\begin{question} How many topological types of locally finite trees with a finite number of rays are there?
\end{question}

For any well-ordered set, its {\it order type} is the only ordinal which is order-equivalent to it.

\begin{question} Is it possible to construct, for any $\alpha \in \omega_1$ a family of topological types of locally finite trees of order type $\alpha$?
\end{question}

We deal with the above questions in \cite{trees}.

\bibliographystyle{plain}
\bibliography{mybib}

\begin{thebibliography}{1}

\bibitem{trees}
J~Bruno and P.~Szeptycki.
\newblock On families of toplogical types of locally finite trees.
\newblock {\em (In preparation)}.

\bibitem{MR2159259}
Reinhard Diestel.
\newblock {\em Graph theory}, volume 173 of {\em Graduate Texts in
  Mathematics}.
\newblock Springer-Verlag, Berlin, third edition, 2005.

\bibitem{MR1816801}
Daniela K{\"u}hn.
\newblock On well-quasi-ordering infinite trees---{N}ash-{W}illiams's theorem
  revisited.
\newblock {\em Math. Proc. Cambridge Philos. Soc.}, 130(3):401--408, 2001.

\bibitem{MR2236511}
Lilian Matthiesen.
\newblock There are uncountably many topological types of locally finite trees.
\newblock {\em J. Combin. Theory Ser. B}, 96(5):758--760, 2006.

\bibitem{MR0175814}
C.~St. J.~A. Nash-Williams.
\newblock On well-quasi-ordering infinite trees.
\newblock {\em Proc. Cambridge Philos. Soc.}, 61:697--720, 1965.

\end{thebibliography}

\end{document}